\documentclass[11pt]{amsart}

\usepackage[a4paper,hmargin=3.5cm,vmargin=4cm]{geometry}
\usepackage{amsfonts,amssymb,amscd,amstext}
\usepackage{graphicx}
\usepackage[dvips]{epsfig}

\usepackage[utf8]{inputenc}
\usepackage{hyperref}

\usepackage{verbatim}
\usepackage{fancyhdr}
\input xy
\xyoption{all}

\usepackage{times}
\usepackage{enumerate}
\usepackage{titlesec}
\usepackage{mathrsfs}

\titleformat{\section}
{\filcenter\bfseries\large} {\thesection{.}}{0.2cm}{}
\titleformat{\subsection}[runin]
{\bfseries} {\thesubsection{.}}{0.15cm}{}[.]
\titleformat{\subsubsection}[runin]
{\em}{\thesubsubsection{.}}{0.15cm}{}[.]

\usepackage[up,bf]{caption}

\newtheorem{theorem}{Theorem}[section]

\newtheorem{lemma}[theorem]{Lemma}

\theoremstyle{definition}

\newtheorem{remark}[theorem]{Remark}

\newtheorem{problem}[theorem]{Problem}

\numberwithin{equation}{section}
\numberwithin{figure}{section}

\newcommand\Lcal{\mathcal{L}}

\newcommand\Cscr{\mathscr{C}}

\newcommand\B{\mathbb{B}}
\newcommand\C{\mathbb{C}}

\newcommand\N{\mathbb{N}}

\newcommand\R{\mathbb{R}}

\newcommand\igot{\mathfrak{i}}

\renewcommand\igot{\mathfrak{i}}

\renewcommand\imath{\igot}

\newcommand\hra{\hookrightarrow}

\newcommand\di{\partial}

\usepackage{color}

\begin{document}

\fancyhead[LO]{Embedding bordered Riemann surfaces in strongly pseudoconvex domains}
\fancyhead[RE]{F.\ Forstneri\v c} 
\fancyhead[RO,LE]{\thepage}

\thispagestyle{empty}

\begin{center}
{\bf\LARGE Embedding bordered Riemann surfaces in \\ strongly pseudoconvex domains}

\vspace*{0.5cm}

{\large\bf  Franc Forstneri{\v c}} 
\end{center}


\vspace*{0.5cm}

\begin{quote}
{\small
\noindent {\bf Abstract}\hspace*{0.1cm}
We show that every bordered Riemann surface, $M$, with smooth boundary $bM$ 
admits a proper holomorphic map $M\to \Omega$ into any bounded strongly pseudoconvex domain $\Omega$ in $\C^n$,
$n>1$, extending to a smooth map $f:\overline M\to\overline \Omega$ which can be chosen an 
immersion if $n\ge 3$ and an embedding if $n\ge 4$. 
Furthermore, $f$ can be chosen to approximate a given holomorphic map $\overline M\to \Omega$ on compacts in 
$M$ and interpolate it at finitely many given points in $M$.

\vspace*{0.2cm}

\noindent{\bf Keywords}\hspace*{0.1cm} Bordered Riemann surface, strongly pseudoconvex domain,
proper holomorphic map, holomorphic embedding.
}

\vspace*{0.1cm}
\noindent{\bf MSC (2010):}\hspace*{0.1cm} 32E30, 32H02

%
%
%
%
%

\vspace*{0.1cm}
\noindent {\bf Date: 4 January 2023} 

\end{quote}

\vspace{5mm}
\centerline{\em In memoriam Mihnea Col{\c t}oiu}
\vspace{5mm}



\section{The main result} \label{sec:intro}

A bordered Riemann surface is an open relatively compact domain, $M$, in another Riemann surface whose nonempty boundary $bM$ is the union of finitely many Jordan curves. 
According to Stout \cite[Theorem 8.1]{Stout1965TAMS}, such a surface is conformally
equivalent to a domain of the form $R\setminus \bigcup_{i=1}^l \Delta_i$ where $R$ is a 
compact Riemann surface and $\Delta_1,\ldots,\Delta_l$ are pairwise disjoint closed discs 
with real analytic boundaries in $R$. 

A perennial problem in complex geometry is to understand the space of 
proper complex curves in a given complex manifold. 
We contribute to this topic with the following result.

%
%
\begin{theorem}\label{th:main1}
Let $\Omega$ be a bounded strongly pseudoconvex domain with boundary of class 
$\Cscr^{k+1}$ for some integer $k\ge 3$ in a complex Euclidean space $\C^n$ for $n>1$, 
and let $M$ be a bordered Riemann surface with real analytic boundary. 
Given a continuous map $f_0:\overline M\to \Omega$ which is holomorphic in $M$, 
a compact set $K$ in $M$, and a number $\epsilon >0$ there exists a proper holomorphic 
map $f:M\to \Omega$, extending to a map $f:\overline M\to\overline \Omega$ of
H\"older class $\Cscr^{k,\alpha}(\overline M)$ for any $\alpha\in (0,1)$
(and holomorphic on a neighbourhood of $\overline M$ if $b\Omega$ is real analytic)
such that $\sup_K |f-f_0|<\epsilon$ and $f$ agrees with $f_0$ to a given order at 
finitely many given points in $M$. If $n\ge 3$ then $f$ can be chosen an immersion on 
$\overline M$, and if $n\ge 4$ then $f$ can be chosen an embedding on $\overline M$ 
provided that the interpolation conditions allow it.
\end{theorem}

Recall that a bounded domain $\Omega$ in $\C^n$ is {\em strongly pseudoconvex} if it is
of the form $\Omega=\{\rho<0\}$, where $\rho$ is a strongly plurisubharmonic function 
on a neighbourhood $U\subset\C^n$ of $\overline \Omega$, i.e., $\rho$ has 
positive definite Levi form: 
\begin{equation}\label{eq:Leviform}
	\Lcal_\rho(z,w) =  \sum_{i,j=1}^n \frac{\partial^2\rho (z)}{\di z_i \di \bar z_j} w_i\overline w_j >0
	\quad\text{for all}\ z\in U\ \text{and}\ w\in \C^n\setminus \{0\},
\end{equation}
and it satisfies the condition $d\rho\ne 0$ on the boundary $b\Omega=\{\rho=0\}$. 

The main new point is that Theorem \ref{th:main1} holds for bordered Riemann surfaces 
of any given conformal type. If one only wishes to control the topological type of $M$, then the conclusion follows 
easily from the known results, and in this case we can find proper holomorphic immersions if $n=2$ and proper holomorphic embeddings if $n\ge 3$.

As pointed in Remark \ref{rem:weaker}, Theorem \ref{th:main1} holds under a weaker geometric assumption on $\Omega$.
It likely holds for any relatively compact, strongly pseudoconvex domain $\Omega$ in 
a Stein manifold $X$ of appropriate dimension; we leave this generalization to an interested reader. 

The Hopf--Ole\u{\i}nik lemma \cite{Hopf1952PAMS,Oleinik1952MS}, applied to the negative 
subharmonic function $\rho\circ f$ on $M$ which vanishes on $bM$, shows that a map $f$ as in 
the theorem is an immersion along the boundary $bM$ which intersects $b\Omega$ transversely. 
Conversely, if $C$ is a smooth curve in $b\Omega$ lying in the boundary of a 
complex analytic curve $A\subset \Omega$ then the pair $(A,C)$ is a local smooth manifold 
with boundary at every point $p\in C$; see \cite{Chirka1982MS} and 
\cite[Theorem 1.1]{Forstneric1988PM}. 

%
%
\begin{remark}\label{rem:generalposition}
A generic holomorphic map $M\to\C^n$ from an open Riemann surface 
is an immersion if $n\ge 2$, and is an injective immersion if $n\ge 3$. 
Hence, the conclusion of the theorem is not be optimal in this respect. 
This loss of the optimal immersion and embedding dimension by one may 
be due to the method of proof. On the other hand, it is not clear whether the 
general position argument gives the optimal conclusion within the 
space of holomorphic maps $M\to\Omega\subset\C^n$ which extend smoothly to 
$\overline M$ and map $bM$ to $b\Omega$. We do not investigate this question here.
\end{remark}

There are many proper holomorphic discs in the ball $\B^n$ of $\C^n$ 
given by rational and also by polynomial maps; see 
\cite{DAngelo2021,DAngeloHuoXiao2017}. We ask the following question.

\begin{problem}
Let $M$ be a bordered Riemann surface with real analytic or real algebraic boundary 
in a compact Riemann surface $R$. Is there a proper holomorphic 
embedding $M\hra\B^n$ for some $n>1$ given by a
meromorphic map of $R$ to $\C^n$ without poles in $\overline M$?
\end{problem}

%
%
Let us place our theorem in the context of known results.

A partial result for the disc $\Delta=\{\zeta\in \C:|\zeta|<1\}$, with interpolation at one point, 
was obtained by Forstneri\v c and Globevnik \cite{ForstnericGlobevnik1992CMH} in 1992.
The proof in this case is much simpler since the disc carries a unique conformal structure. 
The main contribution of \cite{ForstnericGlobevnik1992CMH} 
was the technique of pushing the boundary of an analytic disc in 
a strongly pseudoconvex domain $\Omega$ into the boundary $b\Omega$
by using the Riemann--Hilbert boundary value problem.  

Without asking for approximation, interpolation or general position properties of the map 
$f$ in Theorem \ref{th:main1}, the result follows from the disc case. Indeed, by 
Ahlfors \cite{Ahlfors1950} every bordered Riemann surface, $M$,  
admits many inner functions, i.e., nonconstant continuous functions on $\overline M$ 
which are holomorphic on $M$ and have modulus one on the boundary $bM$. 
Any such function gives a proper holomorphic map $M\to\Delta$. 
The smallest degree of such a map is $2g+m$ where $g$ is the genus of $M$ 
and $m$ is the number of its boundary components. If the boundary $bM$ is real analytic 
then, by the reflection principle, every inner function on $M$ extends to a 
holomorphic function on a neighbourhood of $\overline M$ in the ambient 
Riemann surface. Composing an inner function $\overline M\to \overline \Delta$ 
with a proper holomorphic map $\Delta \to \Omega$ extending smoothly to 
$\overline \Delta$ gives the result.

For maps to the polydisc there is a result of Stout from 1966 
\cite[IV.1 Theorem]{Stout1966MZ} saying the following. 
Let $T=b\Delta=\{\zeta\in \C:|\zeta|=1\}$. 
On every bordered Riemann surface, $M$, there exist three inner functions $f_1,f_2,f_3$ 
such that the map $f=(f_1,f_2,f_3):\overline M\to \overline\Delta^3$ into the closed 
$3$-polydisc separates the points in $\overline M$, it maps the boundary $bM$ to the torus 
$T^3$ (the distinguished boundary of the polydisc $\Delta^3$), and $f:M\hra \Delta^3$ is a 
holomorphic embedding. Including $\Delta^3$ in the ball centred at $0$ 
of radius $\sqrt 3$, the torus $T^3$ sits in the sphere, so this also 
gives proper holomorphic embeddings of $M$ into the 3-ball. However, approximation 
and interpolation of maps as in Theorem \ref{th:main1} are impossible by this method.

In 2004 it was shown by \v Cerne \cite[Corollary 1.4]{Cerne2004AJM} that 
every finitely bordered planar domain with smooth boundary embeds properly 
holomorphically and smoothly up to the boundary in any smoothly bounded convex 
domain in $\C^n$ for $n\ge 2$. The technique developed in his paper 
(see also the sequel \cite{CerneFlores2007TAMS} by \v Cerne and Flores) is the basis
for our construction in the present paper. It uses solutions to certain Riemann--Hilbert
boundary value problems on $M$, similarly to what was done by the author in 
\cite{Forstneric1988IUMJ} (1988) and used in \cite{ForstnericGlobevnik1992CMH} in 
the special case when $M$ is the disc $\Delta$. 

More recently, Coupet, Sukhov, and Tumanov \cite{CoupetSukhovTumanov2009} 
constructed proper $J$-holomorphic discs in almost Stein manifolds $(\Omega,J)$ 
of real dimension $4$. Their approach follows the same geometric scheme of reducing 
to a Riemann--Hilbert problem, but the details are considerably more difficult in the 
case of a nonintegrable almost complex structure $J$.

For general bordered Riemann surfaces $M$ and strongly pseudoconvex domains 
$\Omega$ as in Theorem \ref{th:main1}, the extant methods on the one hand provide 
proper holomorphic maps $M\to \Omega$ extending to continuous maps 
$\overline M\to\overline \Omega$ 
(see \cite[Theorem 8.3.13]{AlarconForstnericLopez2021}, which also gives an 
asymptotic estimate of the distance between the initial holomorphic map 
$f_0:\overline M\to\Omega$ and the final map $f:\overline M\to \overline \Omega$ 
with $f(bM)\subset b\Omega$ in terms of the distance from $f_0(bM)\subset \Omega$ 
to $b\Omega$), and on the other hand give maps as in Theorem \ref{th:main1} but 
with a slightly changed conformal structure on the underlying surface $M$. 
More precisely, if $M\subset M'$ are bordered Riemann surfaces 
such that $M$ is relatively compact in $M'$ and smoothly diffeotopic to $M'$,
there are a bordered Riemann surface $M_0$ diffeomorphic to $M$, 
with $M\subset M_0\subset M'$, and a proper holomorphic map $M_0\to \Omega$ 
as in Theorem \ref{th:main1} extending to a smooth map 
$\overline M_0\to\overline \Omega$. This can be done by using the main result of 
\cite{DrinovecForstneric2007DMJ} as will become clear in the proof of 
Theorem \ref{th:main1}, given in Section \ref{sec:proof}.

For domains $\Omega\subset\C^n$ of dimension $n>2$, Theorem \ref{th:main1} 
holds under weaker geometric (pseudoconvexity) assumptions on $b\Omega$; 
see Remark \ref{rem:weaker}. An indication that this is possible is seen from the 
paper \cite{DrinovecForstneric2007DMJ} by Drinovec Drnov\v sek and Forstneri\v c. 
In this paper, the authors constructed proper holomorphic immersions of any bordered 
Riemann surface into an arbitrary complex manifold $\Omega$ of dimension $n\ge 2$ 
(embeddings if $n\ge 3$) with a smooth exhaustion function $\rho:\Omega\to\R_+$ 
whose Levi form has at least two positive eigenvalues at every point. 
By a different technique, using holomorphic peak functions, their construction was 
extended in \cite{DrinovecForstneric2010AJM} to proper holomorphic 
maps and embeddings of higher dimensional strongly pseudoconvex domains to 
$q$-convex manifolds for suitable values of $q$. (These are complex manifolds satisfying 
a condition on the minimum number of positive eigenvalues of the Levi form of an 
exhaustion function; see Andreotti and Grauert \cite{AndreottiGrauert1962} and the 
survey by Col{\c t}oiu \cite{Coltoiu1995}.) Earlier result  for maps to $q$-convex domains 
in $\C^n$ are due to Dor \cite{Dor1995}.
The main difference from Theorem \ref{th:main1} is that, in the results mentioned above, 
there is no boundary condition to be taken care of. For recent surveys, 
see \cite[Chapter 9]{Forstneric2017E} and \cite{Forstneric2018Korea}.

Finally, we mention the most recent work by Drinovec Drnov\v sek
and Forstneri\v c \cite{DrinovecForstneric2023} who combined the peak function
technique with Oka theory to construct proper holomorphic embeddings
$X\hra\C^n$ from an arbitrary Stein manifold $X$ with $2\dim X<n$ such that the
image lies in a given concave domain $\Omega\subset \C^n$ which is only slightly bigger 
than a halfspace, and we can approximate a given holomorphic embedding 
$K\hra \Omega$ on a compact holomorphically convex subset of $X$.

Theorem \ref{th:main1} has no analogue if $M$ is a strongly 
pseudoconvex domain of higher dimension due to rigidity of CR maps. 
Indeed, by \cite[Theorem 2.2]{Forstneric1986TAMS} the set of germs 
of smooth strongly pseudoconvex real hypersurfaces in $\C^n$ 
which admit a smooth (or even just a formal) Cauchy--Riemann (CR) embedding 
into a sphere of any dimension (the boundary of the ball in some $\C^N$) 
is of the first category in a suitable Fr\'echet space topology. 
The same holds for all strongly pseudoconvex 
real algebraic hypersurfaces as targets; see \cite{Forstneric2004MM}. 
(Recall that a CR map is one that satisfies the tangential Cauchy--Riemann equations.
In the case of a strongly pseudoconvex source hypersurface, such a map extends locally 
holomorphically to the pseudoconvex side.) The question which (strongly) pseudoconvex 
hypersurfaces admit local CR embeddings into various model hypersurfaces has 
attracted considerable attention; see the surveys by Baouendi, Ebenfelt and Rothschild 
\cite{BaouendiEbenfeltRothschild1999}, Pinchuk, Shafikov, and Sukhov 
\cite{PinchukShafikovSukhov2017}, and (for maps to balls) by D'Angelo
\cite{DAngelo2021}. On the other hand, there are proper holomorphic maps and embeddings
of strongly pseudoconvex domains to balls and other model domains which extend 
continuously to the closure of the domain; see the papers by L\o w \cite{Low1985}, 
Globevnik \cite{Globevnik1987}, Hakim \cite{Hakim1990}, and Dor \cite{Dor1990}, among others. 

There are also results in the literature in which one deforms either the domain or the codomain.
A famous example of the latter is the theorem, due to 
Forn{\ae}ss \cite{Fornaess1976AJM} and Chirka and Henkin \cite{HenkinChirka1975}, 
which says that every relatively compact strongly pseudoconvex domain in a Stein 
manifold embeds properly holomorphically and smoothly up to the boundary in some 
strongly convex domain in $\C^N$ for a sufficiently big $N$. 
Concerning the former, for every pair of Stein manifolds $(X,J_X)$ and 
$(Y,J_Y)$ with $\dim Y>2\dim X$ there is a Stein structure $J$ on $X$, homotopic 
to the original structure $J_X$, such that $(X,J)$ admits a proper holomorphic embedding 
into $(Y,J_Y)$ (see \cite[Theorem 5.1]{Forstneric2018Korea}). This is based on the 
work of Eliashberg \cite{Eliashberg1990} and Forstneri\v c and Slapar
\cite{ForstnericSlapar2007MZ,ForstnericSlapar2007MRL}; 
see also the monograph \cite{CieliebakEliashberg2012} by Cieliebak and Eliashberg.

%
%
%

%
%
\section{Proof of Theorem \ref{th:main1}} \label{sec:proof}

Recall that $\Delta=\{\zeta\in \C:|\zeta|<1\}$ denotes the open unit disc. 

Let $\Omega$ be a bounded domain in $\C^n$ for $n>1$ with a smooth strongly plurisubharmonic
defining function $\rho$ on a neighbourhood of $\overline \Omega$. 
Let $z=(z_1,z_2,\ldots,z_n)$ denote the coordinates on $\C^n$. We shall need the following lemma.

%
%
\begin{lemma}\label{lem:discs}
There is an open neighbourhood $U$ of $b\Omega$ and a smooth family of embedded holomorphic 
discs $g_{z,v}: \Delta\to\C^n$ for $z\in U\cap \Omega$ and $v=(v_1,\ldots,v_n)$ a unit vector in $\C^n$ satisfying
\[
	\di\rho(z) v := \sum_{i=1}^n \frac{\di \rho}{\di z_i}(z) \, v_i =0
\]
such that $g_{z,v}(0)=z$, $g'_{z,v}(0)=c v$ for some $c>0$ independent of $(z,v)$, 
and $g_{z,v}(\Delta)$ intersects $b\Omega$ transversely in a closed smooth Jordan curve. 
\end{lemma}

\begin{proof}
This construction can be found in several papers cited in the introductory section (see e.g.\ \cite{ForstnericGlobevnik1992CMH} or \cite[Sect.\ 3]{DrinovecForstneric2012IUMJ}). 
Nevertheless, we given an outline of proof for completeness.
Fix a point $z\in \C^n$ in the domain of $\rho$. The Taylor expansion of $\rho$ around this point equals 
\begin{equation} \label{eq:Taylor}
	\rho(z+w) = \rho(z) +\Re\biggl( 2\sum_{j=1}^n \frac{\di\rho(z)}{\di z_j} w_j  +  
	 \sum_{j,k=1}^n \frac{\partial^2 \rho(z)}{\partial z_j \partial z_k} w_j w_k   \biggr)
	+ \Lcal_\rho(z,w) + o(|w|^2), 
\end{equation}
where $\Lcal_\rho$ is the Levi form \eqref{eq:Leviform}. 
If $z$ is close enough to $b\Omega=\{\rho=0\}$ then the differential $\di\rho(z)$ is 
nonvanishing. For such $z$, the set
\begin{equation}\label{eq:Sigma-z}
	\Sigma_{z}= \Bigl\{ w=(w_1,\ldots,w_n) \in\C^n:
	2\sum_{j=1}^n \frac{\di\rho(z)}{\di z_j} w_j  +  
	 \sum_{j,k=1}^n \frac{\partial^2 \rho(z)}{\partial z_j \partial z_k} w_j w_k =0 \Bigr\}
\end{equation}
is a quadratic complex hypersurface which is smooth (nonsingular) near the origin 
$0\in \C^n$, and it follows from \eqref{eq:Taylor} that near $w=0$ we have that 
\[ 
	\rho(z+w) = \rho(z) +  \Lcal_\rho(z,w) + o(|w|^2) \quad \text{for} \ w\in \Sigma_z.
\] 
Since $\Lcal_\rho(z,w)\ge C |w|^2$ for some constant $C>0$ which can be chosen 
independent of the point $z$ in a neighbourhood of $b\Omega$, we see that $\rho$ 
restricted to the affine complex hypersurface $W_z=z+\Sigma_z$ (passing through $z$)
increases quadratically in a neighbourhood of $z$. Hence, if a point $z\in\Omega$ is 
close enough to $b\Omega$ then the connected component $W^0_z$ of 
$W_z\cap \Omega$ containing $z$ is a topological ball whose smooth 
boundary $bW^0_z$ is a connected component of $W_z\cap b\Omega$ and the 
intersection is transverse along $bW^0_z$. Clearly, $W_z$ is locally near $z$ a graph 
over the affine tangent plane $T_z W_z=z+\ker\di\rho(z)$. This 
gives holomorphic discs $g_{z,v}$ with images in $W_z$ (graphs over affine linear discs in 
the hyperplane $z+\ker\di\rho(z)$ in directions $v\in \ker\di\rho(z)$) satisfying the 
conclusion of the lemma. 

Note that the hypersurfaces $\Sigma_z$ in \eqref{eq:Sigma-z} depend on the first and the second order 
partial derivatives of the defining function $\rho$. Hence, if $\rho$ is of class $\Cscr^{k+1}$ then the holomorphic discs
$g_{z,v}$ obtained in this way are of class $\Cscr^{k-1}$ in the parameters $(z,v)$.
\end{proof}

%
%
\begin{proof}[Proof of Theorem \ref{th:main1}]
Let $f_0:\overline M\to \Omega$ be a continuous map which is holomorphic on $M$, 
and let $K$ be a compact subset of $M$. By \cite[Theorem 1.1]{DrinovecForstneric2007DMJ} 
we can approximate $f_0$ uniformly on $K$ and interpolate it to any given order at finitely 
many given points in $M$ by a holomorphic map $f_1:\overline M\to\Omega$ satisfying 
$f_1(bM)\subset U\cap \Omega$, where $U$ is a neighbourhood of $b\Omega$
for which Lemma \ref{lem:discs} holds and $\di \rho(z)\ne 0$ for each $z\in U$.
By a general position argument we may further assume that $f_1$ is an immersion if $n=2$ 
and an embedding (injective immersion) if $n\ge 3$ provided that this is consistent with the interpolation conditions.

Consider the complex vector bundle $E=\ker \di \rho\to U$,
a complex hyperplane subbundle of class $\Cscr^k$ of the holomorphic tangent bundle $TU$. 
Note that the pull-back bundle $(f_1)^* E\to bM$ by the map $f_1:bM\to U\cap \Omega$
is trivial (since every orientable vector bundle over a circle is trivial and $bM$ is a 
union of finitely many circles). This gives a $\Cscr^k$ map $v:bM\to\C^n$ 
such that for every $x\in bM$ we have  that 
\[
	\text{$|v(x)|=1$\ \ and\ \ $v(x) \in E_{f_1(x)}=\ker \di\rho(f_1(x))$.}
\]

Let $g_{z,v}: \Delta\to\C^n$ be the family of holomorphic discs given by Lemma \ref{lem:discs}. 
With the map $v$ as above, we define the map $h:bM\times \Delta\to\C^n$ by 
\begin{equation}\label{eq:h}
	h(x,\zeta) = g_{f_1(x),v(x)}(\zeta)\ \ \text{for}\ x\in bM\ \text{and}\ \zeta\in \Delta.
\end{equation}
Note that $h$ is smooth of class $\Cscr^{k-1}$ in both variables, and for each $x\in bM$ 
we have that $h(x,0)=f_1(x)$ and the map $h(x,\cdotp):\Delta\to\C^n$ is holomorphic. 
By the construction, each holomorphic disc $g_{z,v}:\Delta\to \C^n$ intersects $b\Omega$
transversely in the image $g_{z,v}(\gamma_{z,v})$ of a closed Jordan curve 
$\gamma_{z,v}\subset \Delta$ enveloping the origin $0\in\Delta$. It follows that 
$h(x,\cdotp):\Delta\to\C^n$ has the same property for every $x\in bM$. 

Consider the Taylor expansion of $h$ in the second variable:
\[ 
	h(x,\zeta) = f_1(x) + \sum_{j=1}^\infty h_j(x) \zeta^j \quad 
	\text{for}\ \ x\in bM\ \text{and}\ \zeta\in\Delta. 
\] 
The $\C^n$-valued functions $h_j$ are of class $\Cscr^{k-1}(bM)$. 
By the Mergelyan theorem we can approximate them in $\Cscr^{k-1}(bM)$ 
by holomorphic functions $\tilde h_j$ on an open collar $V\subset R$ around $bM$ 
in the ambient compact Riemann surface $R$. 

By a theorem of Royden \cite[Theorem 10]{Royden1967JAM},
every holomorphic function on an open neighbourhood of a compact set $L$ in 
a compact Riemann surface $R$ 
can be approximated uniformly on $L$ by meromorphic functions on $R$ without poles on $L$. 
(For more precise theorems of this kind, see the surveys \cite{FornaessForstnericWold2020} and 
\cite[Sect.\ 1.12]{AlarconForstnericLopez2021}.) Applying Royden's result to the functions 
$\tilde h_j$ on a compact neighbourhood $L\subset V$ of $bM$ gives a neighbourhood 
$M'\subset R$ of $\overline M$ and a meromorphic map on $M'\times \C$ with values 
in $\C^n$ of the form
\begin{equation}\label{eq:H}
	H(x,\zeta) = f_1(x) + \sum_{j=1}^N  H_j(x) \zeta^j,\quad x\in M',\ \zeta\in\C
\end{equation}
for some large $N\in\N$ such that the coefficients $H_j$ have no poles on $L\cap M'$ 
and $H$ approximates $h$ \eqref{eq:h} as closely as desired in $\Cscr^{k-1}(bM\times r\Delta)$ 
for any given $r\in (0,1)$. If the approximation of $h$ by $H$ is close enough, 
there is for each $x\in bM$ a closed smooth Jordan curve $\gamma_x\subset r\Delta$ 
bounding a simply connected domain $0\in D_x\subset r\Delta$ such that 
\begin{equation}\label{eq:propH}
	H(x,D_x) \subset \Omega\ \ \text{and}\ \ H(x,\gamma_x) \subset b\Omega 
	\ \ \ \text{for all}\ \ x\in bM.
\end{equation}
By the construction, the family of Jordan curves $\{\gamma_x\}_{x\in bM}$ in $\Delta$ depends
smoothly of class $\Cscr^{k+1}$ (the smoothness class of $\rho$) on the point $x\in bM$.
More precisely, the torus $\{(x,\gamma_x):x\in bM\}$ in $bM\times\Delta$ fibred by these 
curves is of class $\Cscr^{k+1}$.

We now apply the main result of \v Cerne's paper \cite{Cerne2004AJM}, according to which 
there is a function $\zeta:\overline M\to\Delta$ of class $\Cscr^{k,\alpha}$ for any $\alpha\in(0,1)$  
which is holomorphic in $M$, it solves the Riemann--Hilbert boundary value problem 
\begin{equation}\label{eq:zeta}
	\zeta(x)\in \gamma_x\ \ \text{for all}\ x\in bM,
\end{equation}
and it vanishes to a given order at any given finite set of points in $M$. More precisely, 
given an effective divisor $D$ on $M$ with finite support, we can choose $\zeta$ as above 
such that $(\zeta)\ge D$. In particular, we may choose $\zeta$ to vanish at all poles in $M$ 
of the coefficients $H_j$ of the map $H$ in \eqref{eq:H} to the same or higher order 
so that the product $\zeta(x)H_j(x)$ is holomorphic on $M$ for every $j\in\{1,\ldots,N\}$. 
Furthermore, by choosing $\zeta$ to have sufficiently many zeros, the normal families 
argument shows that $\zeta$ will be arbitrarily uniformly close to zero on any given 
compact set in $M$. For some such $\zeta$, the map $f:\overline M\to\C^n$ defined by
\begin{equation}\label{eq:f}
	f(x) = H(x,\zeta(x))\ \ \text{for}\ x\in\overline M
\end{equation}
satisfies the conclusion of the theorem. Indeed, $f\in\Cscr^{k,\alpha}(\overline M)$, 
it is holomorphic in $M$, and for every $x\in bM$ we have that $f(x)\in b\Omega$ by
\eqref{eq:propH} and \eqref{eq:zeta}. Furthermore, if $\zeta$ is chosen to vanish to order $d$ at 
the points $a_1,\ldots,a_l\in M$ which are not the poles of any $H_j$ in \eqref{eq:H}, 
then at these points $f$ 
agrees with $f_1$ to order $d$. At the poles of $H_j$ we can ensure the same condition
by choosing $\zeta$ to vanish to a suitably high order. Next, if $\zeta$ is close enough 
to $0$ on a given compact set $K\subset M$ then $f$ is $\epsilon$-close to $f_1$ on $K$ 
by the construction. Finally, we can ensure that $f(\overline M)$ lies in the domain of the 
strongly plurisubharmonic defining function $\rho$ for $\Omega$, and hence the maximum 
principle applied to the subharmonic function $\rho\circ f$, along with the fact this 
function vanishes on $bM$, implies $f(M)\subset \Omega$.

It was already mentioned in the introduction that a proper holomorphic map 
$f:M\to \Omega$ which is of class $\Cscr^1(\overline M)$ is an immersion along $bM$. 
It remains to explain how to ensure that $f$ is an immersion on $M$ if $n=3$, 
and an embedding if $n\ge 4$. 

Assume that $n\ge 3$. Since $\dim_\R (bM\times \Delta)=3$ and $f_1:\overline M\to \Omega$ 
is an immersion, we can slightly perturb the map $h$ in \eqref{eq:h}, keeping it fixed on 
$bM\times \{0\}$, so as to make it an immersion. This property is inherited by the map 
$H$ in \eqref{eq:H} on the set $V\times r\Delta$ provided that the approximation of $h$ 
by $H$ is close enough and the neighbourhood $V$ of $bM$ is chosen small enough. 
From the definition \eqref{eq:f} it follows that $f$ is then an immersion on $\overline M\cap V$ 
for every choice of the function $\zeta$ in \eqref{eq:zeta}. 
By choosing $\zeta$ to be close enough to zero on the compact set 
$K=M\setminus V\subset M$, we ensure that $f$ is so close to the immersion $f_1$ on $K$ 
that it is itself an immersion there. Hence, for such $\zeta$ the map $f$ is an immersion.

If $n\ge 4$ then we can perturb the map $h$ in \eqref{eq:h} to ensure that it is an embedding. 
As before, this property is inherited by the map $H$ in \eqref{eq:H} on the set 
$V\times r\Delta$ provided that the approximation is close enough and the neighbourhood 
$V$ of $bM$ is chosen small enough. We conclude the proof as in the previous case, 
showing that $f$ is an embedding provided that the function $\zeta$ is close enough to 
$0$ on the compact set $K=M\setminus V\subset M$.
\end{proof}
 
%
%
\begin{remark}\label{rem:weaker}
It is evident from our proof of Theorem \ref{th:main1} that the same result holds 
under the following weaker condition on $\Omega$:
\begin{enumerate}[\rm (a)]
\item there is a smooth defining function $\rho$ for $\Omega$ whose Levi form 
$\Lcal_\rho(z,\cdotp)$ has at least two positive eigenvalues at every point 
$z\in\overline \Omega$, and 
\item for every closed path $\gamma:[0,1]\to b\Omega$ there is a path of nonzero 
complex tangent vectors $0\ne v(t)\in \ker\di\rho(\gamma(t)))=T^\C_{\gamma(t)}b\Omega$ 
such that $\Lcal_\rho(\gamma(t),v(t))>0$ for all $t\in[0,1]$.
\end{enumerate}
If $\Omega$ is strongly pseudoconvex then condition (b) holds since the complex 
vector bundle $\ker\di\rho \to b\Omega$ is orientable, and hence trivial over 
any loop in $b\Omega$. This is an open condition, so it also holds for loops
in $\Omega$ close enough to $b\Omega$; note that this was used in the proof.
\end{remark}


\subsection*{Acknowledgements}
Research was supported by the European Union 
(ERC Advanced grant HPDR, 101053085) 
and grants P1-0291, J1-3005, and N1-0237 from ARRS, Republic of Slovenia. 
I wish to thank Antonio Alarc\'on, Miran \v Cerne, and Barbara Drinovec Drnov\v sek 
for helpful discussions, and Edgar Lee Stout for having provided some relevant references.



\begin{thebibliography}{10}

\bibitem{Ahlfors1950}
L.~V. Ahlfors.
\newblock Open {R}iemann surfaces and extremal problems on compact subregions.
\newblock {\em Comment. Math. Helv.}, 24:100--134, 1950.

\bibitem{AlarconForstnericLopez2021}
A.~Alarc\'{o}n, F.~Forstneri\v{c}, and F.~J. L\'{o}pez.
\newblock {\em Minimal surfaces from a complex analytic viewpoint}.
\newblock Springer Monographs in Mathematics. Springer, Cham, 2021.

\bibitem{AndreottiGrauert1962}
A.~Andreotti and H.~Grauert.
\newblock Th\'eor\`eme de finitude pour la cohomologie des espaces complexes.
\newblock {\em Bull. Soc. Math. France}, 90:193--259, 1962.

\bibitem{BaouendiEbenfeltRothschild1999}
M.~S. Baouendi, P.~Ebenfelt, and L.~P. Rothschild.
\newblock {\em Real submanifolds in complex space and their mappings},
  volume~47 of {\em Princeton Mathematical Series}.
\newblock Princeton University Press, Princeton, NJ, 1999.

\bibitem{Cerne2004AJM}
M.~\v{C}erne.
\newblock Nonlinear {R}iemann-{H}ilbert problem for bordered {R}iemann
  surfaces.
\newblock {\em Amer. J. Math.}, 126(1):65--87, 2004.

\bibitem{CerneFlores2007TAMS}
M.~\v{C}erne and M.~Flores.
\newblock Generalized {A}hlfors functions.
\newblock {\em Trans. Amer. Math. Soc.}, 359(2):671--686, 2007.

\bibitem{Chirka1982MS}
E.~M. Chirka.
\newblock Regularity of the boundaries of analytic sets.
\newblock {\em Mat. Sb. (N.S.)}, 117(159)(3):291--336, 431, 1982.

\bibitem{CieliebakEliashberg2012}
K.~Cieliebak and Y.~Eliashberg.
\newblock {\em From {S}tein to {W}einstein and back. Symplectic geometry of
  affine complex manifolds}, volume~59 of {\em Amer. Math. Soc. Colloquium
  Publications}.
\newblock Amer. Math. Soc., Providence, RI, 2012.

\bibitem{Coltoiu1995}
M.~Col{\c{t}}oiu.
\newblock {$q$}-convexity. {A} survey.
\newblock In {\em Complex analysis and geometry ({T}rento, 1995)}, volume 366
  of {\em Pitman Res. Notes Math. Ser.}, pages 83--93. Longman, Harlow, 1997.

\bibitem{CoupetSukhovTumanov2009}
B.~Coupet, A.~Sukhov, and A.~Tumanov.
\newblock Proper {$J$}-holomorphic discs in {S}tein domains of dimension 2.
\newblock {\em Amer. J. Math.}, 131(3):653--674, 2009.

\bibitem{DAngelo2021}
J.~P. D'Angelo.
\newblock {\em Rational sphere maps}, volume 341 of {\em Progress in
  Mathematics}.
\newblock Birkh\"{a}user/Springer, Cham, 2021.

\bibitem{DAngeloHuoXiao2017}
J.~P.~D'Angelo, Z.~Huo, and M.~Xiao. 
\newblock Proper holomorphic maps from the unit disc to some unit ball. 
\newblock {\em Proc. Amer. Math. Soc.}, 145(6):2649--2660, 2017.

\bibitem{Dor1990}
A.~Dor.
\newblock Proper holomorphic maps between balls in one co-dimension.
\newblock {\em Ark. Mat.}, 28(1):49--100, 1990.

\bibitem{Dor1995}
A.~Dor.
\newblock Immersions and embeddings in domains of holomorphy.
\newblock {\em Trans. Amer. Math. Soc.}, 347(8):2813--2849, 1995.

\bibitem{DrinovecForstneric2007DMJ}
B.~Drinovec~Drnov{\v{s}}ek and F.~Forstneri\v{c}.
\newblock Holomorphic curves in complex spaces.
\newblock {\em Duke Math. J.}, 139(2):203--253, 2007.

\bibitem{DrinovecForstneric2010AJM}
B.~Drinovec~Drnov{\v{s}}ek and F.~Forstneri\v{c}.
\newblock Strongly pseudoconvex domains as subvarieties of complex manifolds.
\newblock {\em Amer. J. Math.}, 132(2):331--360, 2010.

\bibitem{DrinovecForstneric2012IUMJ}
B.~Drinovec~Drnov{\v{s}}ek and F.~Forstneri\v{c}.
\newblock The {P}oletsky-{R}osay theorem on singular complex spaces.
\newblock {\em Indiana Univ. Math. J.}, 61(4):1407--1423, 2012.

\bibitem{DrinovecForstneric2023}
B.~Drinovec~Drnov{\v{s}}ek and F.~Forstneri\v{c}.
\newblock Proper holomorphic maps in Euclidean spaces avoiding unbounded convex sets.
\newblock {\em arXiv e-prints}, January 2023.
\newblock \url{http://arxiv.org/abs/2301.01268} 

\bibitem{Eliashberg1990}
Y.~Eliashberg.
\newblock Topological characterization of {S}tein manifolds of dimension
  {$>2$}.
\newblock {\em Internat. J. Math.}, 1(1):29--46, 1990.

\bibitem{Fornaess1976AJM}
J.~E. Fornaess.
\newblock Embedding strictly pseudoconvex domains in convex domains.
\newblock {\em Amer. J. Math.}, 98(2):529--569, 1976.

\bibitem{FornaessForstnericWold2020}
J.~E. {Forn{\ae}ss}, F.~{Forstneri\v{c}}, and E.~{Wold}.
\newblock {Holomorphic approximation: the legacy of Weierstrass, Runge,
  Oka-Weil, and Mergelyan.}
\newblock In {\em {Advancements in complex analysis. From theory to practice}},
  pages 133--192. Cham: Springer, 2020.

\bibitem{Forstneric1986TAMS}
F.~Forstneri\v{c}.
\newblock Embedding strictly pseudoconvex domains into balls.
\newblock {\em Trans. Amer. Math. Soc.}, 295(1):347--368, 1986.

\bibitem{Forstneric1988IUMJ}
F.~Forstneri\v{c}.
\newblock Polynomial hulls of sets fibered over the circle.
\newblock {\em Indiana Univ. Math. J.}, 37(4):869--889, 1988.

\bibitem{Forstneric1988PM}
F.~Forstneri\v{c}.
\newblock Regularity of varieties in strictly pseudoconvex domains.
\newblock {\em Publ. Mat.}, 32(1):145--150, 1988.

\bibitem{Forstneric2004MM}
F.~Forstneri\v{c}.
\newblock Most real analytic {C}auchy-{R}iemann manifolds are nonalgebraizable.
\newblock {\em Manuscripta Math.}, 115(4):489--494, 2004.

\bibitem{Forstneric2017E}
F.~Forstneri\v{c}.
\newblock {\em Stein manifolds and holomorphic mappings (The homotopy principle
  in complex analysis)}, volume~56 of {\em Ergebnisse der Mathematik und ihrer
  Grenzgebiete. 3. Folge}.
\newblock Springer, Cham, second edition, 2017.

\bibitem{Forstneric2018Korea}
F.~Forstneri\v{c}.
\newblock Holomorphic embeddings and immersions of {S}tein manifolds: a survey.
\newblock In {\em Geometric complex analysis}, volume 246 of {\em Springer
  Proc. Math. Stat.}, pages 145--169. Springer, Singapore, 2018.

\bibitem{ForstnericGlobevnik1992CMH}
F.~Forstneri\v{c} and J.~Globevnik.
\newblock Discs in pseudoconvex domains.
\newblock {\em Comment. Math. Helv.}, 67(1):129--145, 1992.

\bibitem{ForstnericSlapar2007MRL}
F.~Forstneri\v{c} and M.~Slapar.
\newblock Deformations of {S}tein structures and extensions of holomorphic
  mappings.
\newblock {\em Math. Res. Lett.}, 14(2):343--357, 2007.

\bibitem{ForstnericSlapar2007MZ}
F.~Forstneri\v{c} and M.~Slapar.
\newblock Stein structures and holomorphic mappings.
\newblock {\em Math. Z.}, 256(3):615--646, 2007.

\bibitem{Globevnik1987}
J.~Globevnik.
\newblock Boundary interpolation by proper holomorphic maps.
\newblock {\em Math. Z.}, 194(3):365--373, 1987.

\bibitem{Hakim1990}
M.~Hakim.
\newblock Applications holomorphes propres continues de domaines strictement
  pseudoconvexes de {${\bf C}^n$} dans la boule unit\'{e} de {${\bf C}^{n+1}$}.
\newblock {\em Duke Math. J.}, 60(1):115--133, 1990.

\bibitem{HenkinChirka1975}
G.~M. Henkin and E.~M. {\v{C}}irka.
\newblock Boundary properties of holomorphic functions of several complex
  variables.
\newblock In {\em Current problems in mathematics, {V}ol. 4 ({R}ussian)}, pages
  12--142. (errata insert). Akad. Nauk SSSR Vsesojuz. Inst. Nau\v cn. i Tehn.
  Informacii, Moscow, 1975.
  
\bibitem{Hopf1952PAMS}
E.~Hopf.
\newblock A remark on linear elliptic differential equations of second order.
\newblock {\em Proc. Amer. Math. Soc.}, 3:791--793, 1952.

\bibitem{Low1985}
E.~L{\o}w.
\newblock Embeddings and proper holomorphic maps of strictly pseudoconvex
  domains into polydiscs and balls.
\newblock {\em Math. Z.}, 190(3):401--410, 1985.

\bibitem{Oleinik1952MS}
O.~A. Ole\u{\i}nik.
\newblock On properties of solutions of certain boundary problems for equations
  of elliptic type.
\newblock {\em Mat. Sbornik N.S.}, 30(72):695--702, 1952.

\bibitem{PinchukShafikovSukhov2017}
S.~Pinchuk, R.~Shafikov, and A.~Sukhov.
\newblock Some aspects of holomorphic mappings: a survey.
\newblock {\em Tr. Mat. Inst. Steklova}, 298(Kompleksny\u{\i} Analiz i ego
  Prilozheniya):227--266, 2017.
\newblock English version published in Proc. Steklov Inst. Math. {{\bf{2}}98}
  (2017), no. 1, 212--247.

\bibitem{Royden1967JAM}
H.~L. Royden.
\newblock Function theory on compact {R}iemann surfaces.
\newblock {\em J. Analyse Math.}, 18:295--327, 1967.

\bibitem{Stout1965TAMS}
E.~L. Stout.
\newblock Bounded holomorphic functions on finite {R}iemann surfaces.
\newblock {\em Trans. Amer. Math. Soc.}, 120:255--285, 1965.

\bibitem{Stout1966MZ}
E.~L. Stout.
\newblock On some algebras of analytic functions on finite open {R}iemann
  surfaces.
\newblock {\em Math. Z.}, 92:366--379, 1966.

\end{thebibliography}



\vspace*{0.5cm}
\noindent Franc Forstneri\v c

\noindent Faculty of Mathematics and Physics, University of Ljubljana, Jadranska 19, SI--1000 Ljubljana, Slovenia

\noindent 
Institute of Mathematics, Physics and Mechanics, Jadranska 19, SI--1000 Ljubljana, Slovenia.

\noindent e-mail: {\tt franc.forstneric@fmf.uni-lj.si}

\end{document}